\documentclass[12pt]{article}

\renewcommand{\baselinestretch}{1.1}
\usepackage[demo]{graphicx} 
\usepackage{hyperref}
\usepackage{float}
\usepackage{indentfirst}
\usepackage{tikz}
\usepackage{amssymb}
\usepackage{amsmath,mathrsfs,amsfonts}
\usepackage{graphics}
\usepackage{amsthm}
\usepackage{setspace}
\usepackage{cases}
\usepackage{setspace}
\usepackage[mathscr]{eucal}
\usepackage{caption}
\usepackage{subcaption}
\usepackage{rotating}
\usetikzlibrary{decorations.pathreplacing}
\newtheorem{theo}{Theorem}
\newtheorem{theorem}{Theorem}
\newtheorem{lem}[theo]{Lemma}
\newtheorem{claim}{Claim}

\newtheorem{conj}[theo]{Conjecture}

\newtheorem{problem}[theo]{Problem}
 \theoremstyle{definition}
\newtheorem{defi}[theo]{Definition}

\theoremstyle{remark}

\newcounter{casenum}[theo]

\newcounter{subcasenum}[theo]

\newcounter{claimnum}[theo]

\allowdisplaybreaks
\renewcommand{\baselinestretch}{1.3}
\begin{document}
\thispagestyle{plain}

\begin{center} {\Large Degree sequences of spectrally extremal connected nonregular graphs with prescribed maximum degree: the dense case
}
\end{center}
\pagestyle{plain}
\begin{center}
	{
		{\small  Zejun Huang, Chenxi Yang \footnote{  Corresponding author. \\ Email: zejunhuang@szu.edu.cn (Huang), yangchenxi2022@email.szu.edu.cn (Yang)}}\\[3mm]
		{\small   School of Mathematical Sciences, Shenzhen University, Shenzhen 518060, China }\\
		
	}
\end{center}
\renewcommand{\baselinestretch}{1.1}
\begin{center}
\begin{minipage}{140mm}

\begin{center}
{\bf Abstract}
\end{center}
{\small Let $G$ be an $n$-vertex connected nonregular graph  with maximum degree $\Delta$ that attains the maximum spectral radius. For $3\le\Delta\le n-2$,
Liu and Li  (2008) conjectured that $G$ has the degree sequence  $(\Delta,\ldots,\Delta,\delta)$ with $\delta<\Delta$ as large as possible.
In this paper, we study the dense case where $\Delta=n-k$ with $k\ge4$ fixed and $n$ sufficiently large. We prove that
\[
d(G)=
\begin{cases}
(\Delta,\ldots,\Delta,  1), & \text{if $n$ is even and $k$ is odd};\\
(\Delta,\ldots,\Delta,\Delta-1), & \text{if $n$ is odd and $k$ is even};\\
(\Delta,\ldots,\Delta,2), & \text{if $n$ and $k$ have the same parity}.
\end{cases}
\]

{\bf Keywords:} nonregular graph;  spectral radius; degree sequence; extremal graph
}

\end{minipage}
\end{center}

\section{Introduction}

Graphs   in this paper are simple. For a graph $G$ of order $n$, we denote its vertex set and edge set by $V(G)$ and $E(G)$, respectively. The cardinalities of $V(G)$ and $E(G)$ are called the {\it order} and the {\it size} of $G$, respectively.  For~a~vertex $u\in V(G)$, the {\it neighborhood} of $u$, denoted $N_G(u)$, is the set of vertices adjacent to $u$. The {\it closed neighborhood} of $u$ is defined as $N_G[u]:=N_G(u)\cup \{u\}$. When no confusion can arise, we write $N(u)$ and $N[u]$ instead of $N_G(u)$ and $N_G[u]$. For $U\subseteq V(G)$ and $v\in V(G)$, let $N_U(v):=N_G(v)\cap U$. For two disjoint vertex subsets $V_1,V_2\subseteq V(G)$, we write $\partial(V_1,V_2)=\{xy\in E(G):x\in V_1,\ y\in V_2\}$.
The degree of a vertex $v$ in $G$ is denoted by $d_G(v)$, or simply  $d(v)$ when the graph $G$ is clear from the context. We denote by $\Delta(G)$, $\delta(G)$ and $\bar d(G)$ the maximum degree, the minimum degree and the average degree of $G$, respectively. The degree sequence of $G$ is the sequence
$d(G)=(d_1,d_2,\ldots,d_n)$, where $d_1\ge d_2\ge \cdots \ge d_n$ are the degrees of the vertices of $G$ arranged in nonincreasing order, counted with multiplicities. We use the notation $(a_1^{m_1},a_2^{m_2},\ldots,a_t^{m_t})$ to indicate that $a_i$ is repeated $m_i$ times for each $i$, and we omit   the exponent when $m_i=1$. The {\it spectral radius} of $G$, denoted $\rho(G)$, is the spectral radius of its adjacency matrix $A(G)$.

It is well known that $\rho(G)\le \Delta(G)$ for all connected graphs $G$, with equality   if and only if $G$ is regular. In 2004  Stevanovi$\rm{\acute{c}}$ \cite{D} asked a natural interesting problem that how far the spectral radius of a graph is from the maximum degree.
Suppose $n\ge3$ and $2\le\Delta\le n-1$.
Let $\mathcal{C}(n,\Delta)$ be the feasible class consisting of all connected nonregular graphs of order $n$ with maximum degree $\Delta$, and define
$$
\lambda_1(n,\Delta):=\max_{H\in\mathcal{C}(n,\Delta)}\rho(H).
$$
Stevanovi$\rm{\acute{c}}$ \cite{D} obtain the first lower bound for $\Delta-\lambda_1(n,\Delta)$, which was improved by a subsequently works; see, for example, \cite{SMC,CGN,FZ,LLL,LHY,LSW,WZ,Zhang} and the references therein.

Beyond estimating the gap $\Delta-\lambda_1(n,\Delta)$, a complementary problem is to determine the degree sequences and exact structures of the extremal graphs.
Let
$$
\mathcal{G}(n,\Delta):=\{H\in\mathcal{C}(n,\Delta):\rho(H)=\lambda_1(n,\Delta)\}
$$
be the family of extremal graphs in $\mathcal{C}(n,\Delta)$.
Liu and Li \cite{LL1} conjectured the following degree-sequence property for the extremal graphs.

\begin{conj}\cite{LL1}\label{conj1}
  Suppose $3\le\Delta\le n-2$ and $G\in\mathcal{G}(n,\Delta)$. Then the degree sequence of $G$ is $(\Delta^{n-1},\delta)$, where
  \begin{equation}
  \nonumber
  \delta=\left\{
  \begin{aligned}
  &\Delta-1,&&\text{if }n\Delta\text{ is odd}; \\
  &\Delta-2,&&\text{if }n\Delta\text{ is even}. \\
  \end{aligned}
  \right.
  \end{equation}
\end{conj}

Liu \cite{LLL} confirmed Conjecture \ref{conj1} for $\Delta\in\{3,4\}$ and determined the exact structures of the corresponding extremal graphs.
In the same paper, Liu proposed the following modification for the case that $\Delta$ is fixed and $n$ tends to infinity.

\begin{conj}\cite{LLL}\label{conj2}
  For each fixed integer $\Delta\ge3$ and all sufficiently large $n$, every graph $G\in\mathcal{G}(n,\Delta)$ has degree sequence $(\Delta^n,\delta)$, where
 \begin{equation*}
  \delta=\left\{
  \begin{aligned}
  &\Delta-1,&&\text{if $\Delta$ and $n$ are both odd};\\
  &1,&&\text{if $\Delta$ is odd and $n$ is even};\\
  &\Delta-2,&&\text{if $\Delta$ is even}.
  \end{aligned}
  \right.
  \end{equation*}
\end{conj}

Conjecture \ref{conj2} concerns the sparse case of Conjecture \ref{conj1} essentially. In this paper,
we consider the  dense case    $\Delta=n-k$ of Conjecture \ref{conj1}, where $k$ is fixed and $n$ is sufficiently large. We formulate our problem as follows.

\begin{problem}
  For each fixed integer $k\ge2$ and all sufficiently large $n$, determine the possible degree sequences of graphs in $\mathcal{G}(n,n-k)$.
\end{problem}

The case $k=1$ is trivial that all graphs in $\mathcal{G}(n,n-1)$ are isomorphic to $K_n-e$ \cite{LL1}.
For $k=2$, equivalently $\Delta=n-2$, Huang et al. \cite{HLY} and Sun \cite{MS} independently characterized the extremal graphs for $n\ge5$.
These structural results imply that each extremal graph $G\in\mathcal{G}(n,n-2)$ has degree sequence
$$
d(G)\in
\begin{cases}
\{(\Delta^{n-1},n-3)\}, & \text{if } n \text{ is odd};\\
\{(\Delta^{n-1},n-4),(\Delta^{n-1},2)\}, & \text{if } n \text{ is even}.
\end{cases}
$$
Moreover, both degree sequences in the second case are attained by extremal graphs when $n$ is even.

For $k=3$, equivalently $\Delta=n-3$, Huang et al. \cite{HLY} characterized the extremal graphs for $n\ge59$, yielding
$$
d(G)=
\begin{cases}
(\Delta^{n-1},1), & \text{if } n \text{ is even};\\
(\Delta^{n-1},2), & \text{if } n \text{ is odd}.
\end{cases}
$$

In this paper, we determine the degree sequences of graphs in $\mathcal{G}(n,n-k)$ for each fixed $k\ge4$ and all sufficiently large $n$.
Our main result is as follows.

\begin{theorem}\label{th1}
Let $k\ge 4$ be a fixed integer. Then for each sufficiently large integer $n$   and $\Delta=n-k$, every graph $G\in \mathcal{G}(n,\Delta)$ has degree sequence
\[
d(G)=
\begin{cases}
(\Delta^{n-1},  1), & \text{if $n$ is even and $k$ is odd};\\
(\Delta^{n-1},\Delta-1), & \text{if $n$ is odd and $k$ is even};\\
(\Delta^{n-1},2), & \text{if $n$ and $k$ have the same parity}.
\end{cases}
\]
\end{theorem}

We will present some preliminaries in Section 2 and give the proof of Theorem \ref{th1} in Section 3 and Section 4.
\section{Preliminaries}

In this section, we introduce notation and preliminary lemmas that will be used.
Throughout the remainder of the paper, we always suppose $k\ge4$ is fixed, $n$ is sufficiently large, and $\Delta=n-k$.

Since every graph in $\mathcal{G}(n,\Delta)$ is connected, the adjacency matrix of such a graph is irreducible.
Therefore, by  the Perron-Frobenius theorem, $A(G)$ has a unique positive unit eigenvector corresponding to $\rho(G)$, which is called the {\it Perron vector} of $G$ and denoted by $X_G=(x_v)_{v\in V(G)}$. We call $A(G)X=\rho(G)X$ the {\it Perron equation} of $G$.
For a finite set $V$ and $v\in V$, let $e_v\in\mathbb{R}^V$ denote the standard unit vector whose $v$-component is one and whose other components are zero.

\begin{lem}\label{tv}\cite{Zhan}
Let $A$ be an $n\times n$ nonnegative symmetric real matrix. Then
$$
\rho(A)=\max\limits_{x\in\mathbb{R}^n,\,\|x\|=1}x^TAx.
$$
\end{lem}

Suppose that a symmetric real matrix $A$ has rows and columns indexed by a finite set $V$, and let $\pi=\{V_1,V_2,\ldots,V_r\}$ be a partition of $V$.
For $1\le i,j\le r$, let $A_{ij}$ be the submatrix of $A$ with row indices in $V_i$ and column indices in $V_j$,
and let $b_{ij}$ be the average row sum of $A_{ij}$.
The matrix $B=(b_{ij})$ is called the {\it quotient matrix} of $A$ with respect to $\pi$.
The partition $\pi$ is {\it equitable} if every block $A_{ij}$ has constant row sum.

\begin{lem}\label{QM}\cite{YOU}
Let $A$ be an irreducible nonnegative symmetric matrix, let $\pi$ be a partition of its index set,
and let $B=(b_{i,j})$ be the corresponding quotient matrix. Then
$$
\rho(A) \geq \rho(B).
$$
If the partition $\pi$ is equitable, then $\rho(A)=\rho(B)$, and each eigenvalue of $B$ is an eigenvalue of $A$.
\end{lem}

For a graph $G$, define
$$
\begin{aligned}
S_G:=\{v\in V(G):d_G(v)<\Delta(G)\},\quad T_G:=\{v\in V(G):d_G(v)=\Delta(G)\},\quad
s_G:=|S_G|.
\end{aligned}
$$
When the underlying graph is clear, we simply write $S,T$, and $s$ instead of $S_G,T_G$, and $s_G$, respectively;
in this case, $N_T(v)$ means $N(v)\cap T$.
Thus, if $G$ has order $n$ and maximum degree $\Delta$, then $d(G)=(\Delta^{n-s},\delta_1,\delta_2,\ldots,\delta_s)$,
where $s=s_G$ and $\delta_i<\Delta$ for $i=1,2,\ldots,s$.
We have the following properties for graphs $G$ in $\mathcal{G}(n,\Delta)$.

\begin{lem}\cite{LL2}\label{lemma1}
Let $3\le\Delta\le n-2$ and $G\in \mathcal{G}(n,\Delta)$. Then $G[S]$ is a complete graph. Moreover,  $x_v\leq x_u$ if and only if $N_T(v)\subseteq N_T(u)$ for all $u,v\in S$.
\end{lem}

\begin{samepage}
Let $S'=\{v\in V(G): s\le d(v)<\Delta\}$, $s_1=|S'|$ and $s_2=s-s_1$.
Then $S'\subseteq S$. By Lemma \ref{lemma1}, we have
\begin{equation}\label{eqh1}
 d(u)=s-1 \quad\text{for all}\quad u\in S\setminus S'.
\end{equation}
Hence, for each $G\in \mathcal{G}(n,\Delta)$, the degree sequence of $G$ can be written as
$$d(G)=(\Delta^{n-s},\delta_1,\delta_2,\ldots,\delta_{s_1},(s-1)^{s_2}),$$
where $s\le \delta_i<\Delta \quad \text{for } i=1,2,\ldots,s_1$.

\end{samepage}

Let $G\in\mathcal{G}(n,\Delta)$ and $u,v\in S$.
It follows from Lemma  \ref{lemma1}  that
$d(u)\le d(v)$ if and only if $x_u\le x_v$ if and only if $N_G[u]\subseteq N_G[v]$.
In particular,
$d(u)=d(v)$ if and only if $x_u=x_v$ if and only if $N_G[u]=N_G[v]$.

\begin{lem}\cite{LLL}\label{lemma3}
Let $3\le\Delta\le n-2$, $G\in \mathcal{G}(n,\Delta)$, and $u\in S,v\in T$. Then $x_u<x_v$.
\end{lem}

\begin{lem}\label{lemma4}
Let $G\in \mathcal{G}(n,n-k)$. Then $s\le k-1$.
\end{lem}

\begin{proof}
To the contrary, suppose  $|S|\ge k$.
Since each vertex $u\in T$ has exactly $k-1$ nonneighbors in $V(G)\setminus\{u\}$, $u$ must be adjacent to at least one vertex in $S$. Hence $T=\bigcup_{v\in S}N_T(v)$.
By Lemma \ref{lemma1}, the sets $N_T(v)$, $v\in S$, are linearly ordered by inclusion. Therefore, there exists a vertex $v\in S$ adjacent to all the vertices in $T$.
Together with Lemma \ref{lemma1}, this implies that $v$ is adjacent to all the vertices in $V(G)\setminus\{v\}$, and hence $d(v)=n-1>\Delta$, a contradiction.
\end{proof}

\begin{lem}\label{connect}
Let $k\ge4$, $n\ge4k-2$  and   $\Delta=n-k$.
Suppose $G$ is a graph of order $n$ with maximum degree $\Delta$.
If $G[S]$ is a complete graph, $|S|\le k-1$, and $\partial(S,T)\neq \emptyset$, then $G$ is connected.
\end{lem}

\begin{proof}
Since $|S|\le k-1$ and $\Delta(G)=\Delta$, we have $|N_T(u)|\ge \Delta-|S|\ge n-2k+1$ for all $u\in T$.
Since $\partial(S,T)\neq\emptyset$, we have $|S|\ge1$ and hence $|T|\le n-1$. Therefore, for any distinct $u,v\in T$,
$$
|N_T(u)|+|N_T(v)|\ge2(n-2k+1)>n-1\ge|T|.
$$
It follows that $N_T(u)\cap N_T(v)\neq\emptyset$, so $G[T]$ is connected.
Combining this with the fact that $G[S]$ is complete and $\partial(S,T)\neq\emptyset$, we conclude that $G$ is connected.
\end{proof}

\begin{lem}\label{ad}
Let $k\ge4$ be a fixed integer. Then
$$
\lim_{n\rightarrow\infty}\frac{\lambda_1(n,n-k)}{n}=1.
$$
\end{lem}

\begin{proof}
Let $G\in\mathcal{G}(n,n-k)$. By Lemma \ref{lemma4}, $s\le k-1$. Hence
$$
n-k\ge\lambda_1(n,n-k)=\rho(G)\ge\bar d(G)\ge\frac{(n-s)(n-k)}{n}\ge\frac{(n-k+1)(n-k)}{n}>n-2k+1.
$$
Dividing by $n$ and letting $n\rightarrow\infty$ proves the result.
\end{proof}

Define
$$
\begin{aligned}
\mathcal{D}:={}&\{(\Delta^{n-s},\Delta-1,(s-1)^{s-1}),(\Delta^{n-s},s+1,(s-1)^{s-1}),(\Delta^{n-s},s,(s-1)^{s-1}):1\le s<k\}\\
&\cup\{(\Delta^{n-s},s^2,(s-1)^{s-2}):2\le s<k\}\cup\{(\Delta^{n-1},\Delta-2),(\Delta^{n-2},(\Delta-1)^2)\}.
\end{aligned}
$$
To prove Theorem \ref{th1}, we first state and prove the following lemma.

\begin{lem}\label{lemma5}
Let $k\ge4$ be fixed, let $n$ be a sufficiently large  integer, and set $\Delta=n-k$.
If $G\in \mathcal{G}(n,\Delta)$, then $d(G)\in \mathcal{D}$.
\end{lem}

\section{Proof of Lemma \ref{lemma5}}

We first give a Perron vector estimate that will also be used in the proof of Theorem \ref{th1}.

\begin{lem}\label{lemma6}
Let $k\ge4$ be fixed, let $n$ be sufficiently large, and let $G\in\mathcal{C}(n,n-k)$.
Denote
$$
\rho=\rho(G),\qquad X=X_G=(x_v)_{v\in V},\qquad M=\max_{v\in T}x_v,\qquad m=\min_{v\in T}x_v.
$$
Suppose that
\begin{itemize}
\item[(i)] $G[S]$ is complete;
\item[(ii)] for every $u,v\in S$, either $N_T(u)\subseteq N_T(v)$ or $N_T(v)\subseteq N_T(u)$.
\end{itemize}
Then
\begin{equation}\label{eq1}
0\le M-m<\frac{km}{\rho}\le\frac{k}{\rho\sqrt{n-s}},
\end{equation}
\begin{equation}\label{eq2}
1-\frac{k}{\rho}<\frac{m}{M}\le1\le\frac{M}{m}<1+\frac{k}{\rho}.
\end{equation}
Moreover,
\begin{equation}\label{eq3}
(d(u)-d(v))m\le(\rho+1)(x_u-x_v)\le(d(u)-d(v))M\quad\text{for all}\quad
 u,v\in S \text{ with }d(u)\ge d(v).
\end{equation}
\end{lem}

\begin{proof}
For simplicity, we denote $x(S):=\sum_{w\in S}x_w$.
We first show that $s\le k-1$. Using assumptions~(i) and~(ii), we obtain $s\le k-1$
by the same argument as in the proof of Lemma~\ref{lemma4}.

For each $u\in S$,   condition (i) and the Perron equation at $u$ imply
$$
(\rho+1)x_u=x(S)+\sum_{w\in N_T(u)}x_w,
$$
and hence
$$
(d(u)+1-s)m+x(S)\le(\rho+1)x_u\le(d(u)+1-s)M+x(S).
$$
If $u,v\in S$ and $d(u)\ge d(v)$, then $|N_T(u)|\ge |N_T(v)|$, and condition (ii) gives $N_T(v)\subseteq N_T(u)$. Therefore, we obtain
$$
(\rho+1)(x_u-x_v)=\sum_{w\in N_T(u)\setminus N_T(v)}x_w,
$$
which proves (\ref{eq3}).

Since each vertex of $T$ has degree $n-k$, we have
$$
\rho\ge \frac{2e(G)}{n}\ge \frac{(n-s)(n-k)}{n}>n-2k+1.
$$
Summing
 \begin{equation}\label{eq4}
(\rho+1)x_u\le(d(u)+1-s)M+x(S)
\end{equation}
 over $u\in S$ and using $d(u)\le n-k-1$, we obtain
$$
(\rho+1-s)x(S)\le s(n-k-s)M.
$$
Substituting this bound into \eqref{eq4}, we get
\begin{equation}\label{eqh5}
x_u\le\frac{n-k-s}{\rho+1-s}M\qquad\text{for all }\quad u\in S.
\end{equation}

Recall that
\begin{equation}\label{eq5}
|V(G)\setminus N_G[v]|=k-1\quad\text{for all}\quad v\in T
\end{equation}
and
\begin{equation}\label{eq6}
(\rho+1)x_v=\sum_{w\in V(G)}x_w-\sum_{w\in V(G)\setminus N_G[v]}x_w\quad\text{for all}\quad v\in V(G).
\end{equation}
Choose $p,q\in T$ such that $x_p=M$ and $x_q=m$.
 Subtracting   $p$ and $q$ into \eqref{eq6}, we obtain
$$
(\rho+1)(M-m)=\sum_{w\in V(G)\setminus N_G[q]}x_w-\sum_{w\in V(G)\setminus N_G[p]}x_w.
$$
Since the Perron vector is positive, we have
$$
(\rho+1)(M-m)\le\sum_{w\in V(G)\setminus N_G[q]}x_w.
$$
By \eqref{eqh5},\eqref{eq5} and the definition of $M$, we have
$$
\sum_{w\in V(G)\setminus N_G[q]}x_w\le(k-1)\max\left\{1,\frac{n-k-s}{\rho+1-s}\right\}M.
$$
Since $s\le k-1$ and $n-2k+1<\rho\le n-k$, for sufficiently large $n$ we have
\begin{eqnarray}
\rho(M-m)&=&\frac{\rho+k}{\rho+1}(\rho+1)(M-m)-k(M-m)\nonumber\\
&\le&\frac{(k-1)(\rho+k)}{\rho+1}\max\left\{1,\frac{n-k-s}{\rho+1-s}\right\}M-k(M-m)\nonumber\\
&=&\bigl(k-1+O_k(n^{-1})\bigr)M-k(M-m)\nonumber\\
&<&kM-k(M-m)=km.\label{eq8a}
\end{eqnarray}
Therefore, we get $M/m<1+k/\rho$  and $m/M>\rho/(\rho+k)>1-k/\rho$, which proves (\ref{eq2}).
Since $X^TX=1$, we have $(n-s)m^2\le1$, which implies $m\le1/\sqrt{n-s}$.
Combining this with  \eqref{eq8a}, we have
$$ 0\le M-m<\frac{km}{\rho}\le\frac{k}{\rho\sqrt{n-s}},$$
which proves (\ref{eq1}).
\end{proof}

\begin{proof}[Proof of Lemma \ref{lemma5}]
Suppose $G\in\mathcal{G}(n,\Delta)$.  Let $\rho,X,M,m$ be the same notation as in Lemma \ref{lemma6}, $V:=V(G)$, $\alpha:=\frac{M}{m}$,
 $x(S):=\sum_{u\in S}x_u$ and $x(T):=\sum_{u\in T}x_u$.
Recall that
$$
\begin{aligned}
&S=\{v\in V:d(v)<\Delta\},\qquad T=\{v\in V:d(v)=\Delta\},\qquad s=|S|,\\
&S'=\{v\in V:s\le d(v)<\Delta\},\qquad s_1=|S'|,\qquad s_2=s-s_1.
\end{aligned}
$$
By \eqref{eqh1}, we have $N_T(u)=\emptyset$ for all $u\in S\setminus S'$.
By Lemma \ref{lemma4}, $s\le k-1$.
By Lemma  \ref{lemma1}, we can apply Lemma \ref{lemma6}  to $G$ to get \eqref{eq1},\eqref{eq2} and \eqref{eq3}.

For a graph $H$, the notation $H+uv-ab$ means that the edge $uv$ is added and the edge $ab$ is deleted; notation with more signed edge pairs is interpreted analogously.

We break the proof into the following claims.

\begin{claim}\label{claim1}
$|d(u)-d(v)|\le k$ for all $u,v\in S'$.
\end{claim}
\begin{proof}
If $|S'|\le1$, the claim is trivial. Otherwise, choose $u,v\in S'$ such that $d(u)=\max_{w\in S'}d(w)$ and $d(v)=\min_{w\in S'}d(w)$.
If $d(u)=d(v)$, there is nothing to prove, so we assume that $d(u)>d(v)$. It suffices to show that $d(u)-d(v)<k+1$.

By Lemma \ref{lemma1} and the definition of $S'$, there exist $a,c\in T$ such that $u\nsim a$ and $v\sim c$.
Since $d(u)\ge d(v)$, by Lemma  \ref{lemma1}  we have
\begin{equation}\label{eq7}
N_T(v)\subseteq N_T(u).
\end{equation}
Consequently, $a\notin N_T(v)$ and $c\in N_T(u)$, so $u,v\in N(c)\setminus N(a)$.
Since $d(a)=d(c)$, the sets $N(a)\setminus N(c)$ and $N(c)\setminus N(a)$ have the same cardinality.
Thus $|N(a)\setminus N(c)|\ge2$, and we may choose $b\in N(a)\setminus N(c)$ with $b\ne c$; automatically, $b\notin\{u,v,a\}$.
We claim that $b\in T$. Otherwise,   $b\in S$ implies $d(b)\ge s$, as $b\in N(a)$. Then by the choice of $u$ and the definition of $S'$ we have $d(b)\le d(u)$.
Lemma  \ref{lemma1} then give $N_T(b)\subseteq N_T(u)$, contradicting $a\in N_T(b)\setminus N_T(u)$.

Consider \textbf{Operation~1}: $\boldsymbol{G'=G+ua-ab+bc-cv}$; see Figure \ref{fig:1}.
In $G'$, the vertices $a,b,c$ retain degree $\Delta$, while $u$ gains one edge and $v$ loses one edge; hence $\Delta(G')=\Delta$, and $v$ remains deficient.
The deficient vertices form a subset of $S$, and since no edge inside $S$ is changed, they induce a complete graph of order at most $k-1$.
If $u$ remains deficient, then $ua$ joins a deficient vertex to a vertex of degree $\Delta$; if $u$ reaches degree $\Delta$, then the unchanged edge $uv$ does so.
Hence Lemma \ref{connect} shows that $G'$ is connected. Thus $G'\in\mathcal{C}(n,\Delta)$, and the extremality of $G$ gives $\rho(G')-\rho(G)\le0$.

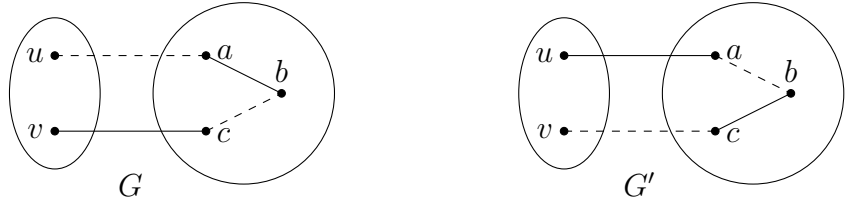
\begin{figure}[H]
	\centering
    \begin{subfigure}{.4\textwidth}
	\centering
	\begin{tikzpicture}
	    \draw  (-1,0) ellipse [x radius=0.6cm, y radius=1cm];
		\draw  (1.5,0) ellipse [x radius=1.2cm, y radius=1.2cm];
        \draw(-1,0.5)[fill] circle[radius=0.5mm];
        \draw(-1,-0.5)[fill] circle[radius=0.5mm];
        \draw(1,0.5)[fill] circle[radius=0.5mm];
        \draw(1,-0.5)[fill] circle[radius=0.5mm];
        \draw(2,0)[fill] circle[radius=0.5mm];
        \draw (1,0.5)--(2,0);\draw (-1,-0.5)--(1,-0.5);
        \draw [dashed](1,-0.5)--(2,0);\draw[dashed](-1,0.5)--(1,0.5);
        \node at(-1,0.5)[left]{$u$};
        \node at(-1,-0.5)[left]{$v$};
        \node at(1,0.55)[right]{$a$};
        \node at(2,0)[above]{$b$};
        \node at(1,-0.55)[right]{$c$};
		\node at(0,-1.5)[above]{$G$};
	\end{tikzpicture}
    \end{subfigure}
    \begin{subfigure}{.4\textwidth}
	\centering
	\begin{tikzpicture}
	    \draw (-1,0) ellipse [x radius=0.6cm, y radius=1cm];
		\draw (1.5,0) ellipse [x radius=1.2cm, y radius=1.2cm];
        \draw (-1,0.5)[fill] circle[radius=0.5mm];
        \draw (-1,-0.5)[fill] circle[radius=0.5mm];
        \draw (1,0.5)[fill] circle[radius=0.5mm];
        \draw (1,-0.5)[fill] circle[radius=0.5mm];
        \draw (2,0)[fill] circle[radius=0.5mm];
        \draw [dashed](1,0.5)--(2,0);\draw [dashed](-1,-0.5)--(1,-0.5);
        \draw (1,-0.5)--(2,0);\draw(-1,0.5)--(1,0.5);
        \node at(-1,0.5)[left]{$u$};
        \node at(-1,-0.5)[left]{$v$};
        \node at(1,0.55)[right]{$a$};
        \node at(2,0)[above]{$b$};
        \node at(1,-0.55)[right]{$c$};
		\node at(0,-1.5)[above]{$G'$};
	\end{tikzpicture}
    \end{subfigure}
	\caption{Operation 1}
	\label{fig:1}
\end{figure}
Since $a,b,c\in T$, Lemma \ref{lemma3} and the definitions of $m,M$ give $x_b>\max\{x_u,x_v\}$, $x_a,x_b\le M$, and $x_c\ge m$.
By Lemma \ref{tv}, we have
\begin{equation}\label{eq8b}
\begin{aligned}
\frac{\rho(G')-\rho(G)}{2}&\ge x_ux_a-x_ax_b+x_bx_c-x_cx_v\\
&=-x_a(x_b-x_u)+x_c(x_b-x_v)\ge -M(x_b-x_u)+m(x_b-x_v)\\
&=(m-M)x_b+(Mx_u-mx_v)\ge (m-M)M+(Mx_u-mx_v)\\
&\ge (m-M)M+m(x_u-x_v).
\end{aligned}
\end{equation}
Since $\rho(G')-\rho(G)\le0$, (\ref{eq8b}) implies
\begin{equation}\label{eq9}
m(x_u-x_v)\le M(M-m).
\end{equation}
Multiplying the left inequality in (\ref{eq3}) by $m$ and using (\ref{eq9}), we obtain
\begin{equation}\label{eq10}
m^2(d(u)-d(v))\le (\rho+1)m(x_u-x_v)\le (\rho+1)M(M-m).
\end{equation}
Since  $\rho\to\infty$ as $n\to\infty$, by (\ref{eq2}) and (\ref{eq10}), we have
\begin{equation}\label{eq11}
d(u)-d(v)\le (\rho+1)\alpha(\alpha-1)<(\rho+1)\frac{k}{\rho}(1+\frac{k}{\rho})<k+1
\end{equation}
for sufficiently large $n$.
Since $d(u)-d(v)$ is an integer, we conclude that $d(u)-d(v)\le k$. This proves the claim.
\end{proof}

\begin{claim}\label{claim2}
If $d(G)\neq (\Delta^{n-2},(\Delta-1)^2)$, then   $x_u+x_v\le M$  for all distinct vertices $u,v\in S$.
\end{claim}

\begin{proof}
If $|S|\le1$, the claim is vacuous. Otherwise, choose distinct vertices $u,v\in S$ such that $d(u)\ge d(v)\ge d(w)$ for all $w\in S\setminus\{u,v\}$.
By Lemma \ref{lemma1}, we have
\begin{equation}\label{eq15}
N_T(w)\subseteq N_T(v)\subseteq N_T(u)\quad\text{and}\quad x_w\le x_v\le x_u\quad\text{for all}\quad w\in S\setminus\{u,v\}.
 \end{equation}
It suffices to prove the inequality for this pair.  Since $G[S]$ is complete and $u$ is deficient, $u$ is not adjacent to every vertex of $T$.
Choose $a\in T\setminus N_T(u)$. Then by \eqref{eq15} we have $N(a)\cap S=\emptyset$, i.e., $N(a)\subseteq T$. Since $d(a)=\Delta$ and $G$ is connected, $N(a)$ contains distinct $b,c\in N(a)\subseteq T$ such that $b\nsim c$.

Consider \textbf{Operation~2}: $\boldsymbol{G'=G+ua+va-ab-ac+bc}$; see Figure \ref{fig:2}.
\begin{figure}[H]
	\centering
    \begin{subfigure}{.4\textwidth}
	\centering
	\begin{tikzpicture}
	    \draw  (-1,0) ellipse [x radius=0.6cm, y radius=1cm];
		\draw  (1.5,0) ellipse [x radius=1.2cm, y radius=1.2cm];
        \draw(-1,0.5)[fill] circle[radius=0.5mm];
        \draw(-1,-0.5)[fill] circle[radius=0.5mm];
        \draw(2,0.5)[fill] circle[radius=0.5mm];
        \draw(2,-0.5)[fill] circle[radius=0.5mm];
        \draw(1,0)[fill] circle[radius=0.5mm];
        \draw [dashed](-1,0.5)--(1,0);\draw [dashed](-1,-0.5)--(1,0);\draw [dashed](2,0.5)--(2,-0.5);
        \draw (1,0)--(2,0.5);\draw(1,0)--(2,-0.5);
        \node at(-1,0.5)[left]{$u$};
        \node at(-1,-0.5)[left]{$v$};
        \node at(2,0.5)[right]{$b$};
        \node at(1,0)[above]{$a$};
        \node at(2,-0.5)[right]{$c$};
		\node at(0,-1.5)[above]{$G$};
	\end{tikzpicture}
    \end{subfigure}
    \begin{subfigure}{.4\textwidth}
	\centering
	\begin{tikzpicture}
	    \draw  (-1,0) ellipse [x radius=0.6cm, y radius=1cm];
		\draw  (1.5,0) ellipse [x radius=1.2cm, y radius=1.2cm];
        \draw(-1,0.5)[fill] circle[radius=0.5mm];
        \draw(-1,-0.5)[fill] circle[radius=0.5mm];
        \draw(2,0.5)[fill] circle[radius=0.5mm];
        \draw(2,-0.5)[fill] circle[radius=0.5mm];
        \draw(1,0)[fill] circle[radius=0.5mm];
        \draw (-1,0.5)--(1,0);\draw (-1,-0.5)--(1,0);\draw (2,0.5)--(2,-0.5);
        \draw [dashed](1,0)--(2,0.5);\draw[dashed](1,0)--(2,-0.5);
        \node at(-1,0.5)[left]{$u$};
        \node at(-1,-0.5)[left]{$v$};
        \node at(2,0.5)[right]{$b$};
        \node at(1,0)[above]{$a$};
        \node at(2,-0.5)[right]{$c$};
		\node at(0,-1.5)[above]{$G'$};
	\end{tikzpicture}
    \end{subfigure}
	\caption{Operation 2}
	\label{fig:2}
\end{figure}
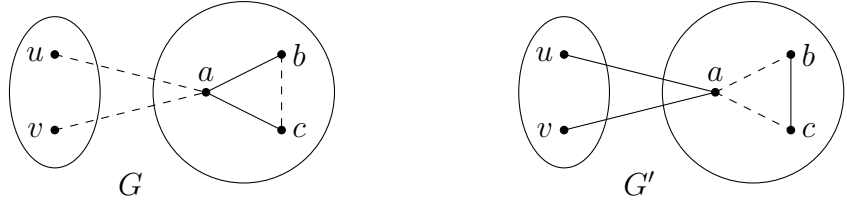
In $G'$, the vertices $a,b,c$ retain degree $\Delta$, while each of $u$ and $v$ gains one edge; hence $\Delta(G')=\Delta$.
The deficient vertices form a subset of $S$, and since no edge inside $S$ is changed, they induce a complete graph of order at most $k-1$.
The graph $G'$ is nonregular; otherwise we have $S=\{u,v\}$ and $d(u)=d(v)=\Delta-1$, which contradicts the condition $d(G)\neq (\Delta^{n-2},(\Delta-1)^2)$.
If one of $u,v$ remains deficient, then $ua$ or $va$ joins a deficient vertex to a vertex of degree $\Delta$; if both reach degree $\Delta$, any remaining deficient vertex is adjacent to them because $G[S]$ is complete.
Hence Lemma \ref{connect} shows that $G'$ is connected. Thus $G'\in\mathcal{C}(n,\Delta)$, and the extremality of $G$ gives $\rho(G')-\rho(G)\le0$.

Since $a,b,c\in T$ and $u,v\in S$, by Lemma \ref{lemma3} we have $x_b>x_u$ and $x_c>x_v$. Moreover, $x_a,x_b,x_c\le M$.
By Lemma \ref{tv}, we obtain
\begin{equation*}\label{eq12}
\begin{aligned}
0\ge \frac{\rho(G')-\rho(G)}{2}&\ge x_ux_a+x_vx_a-x_bx_a-x_cx_a+x_bx_c\\
&=-x_a(x_b+x_c-x_u-x_v)+x_bx_c\\
&\ge -M(x_b+x_c-x_u-x_v)+x_bx_c\\
&=(M-x_b)(M-x_c)+M(x_u+x_v)-M^2\\
&\ge M(x_u+x_v)-M^2,
\end{aligned}
\end{equation*}
which leads to $x_u+x_v\le M$.
\end{proof}

\begin{claim}\label{claim3}
If $G$ has a vertex of degree $\Delta-1$, then $d(G)=(\Delta^{n-2},(\Delta-1)^2)$ or
$d(G)=(\Delta^{n-s},\Delta-1,(s-1)^{s-1})$ with $1\le s\le k-1$.
\end{claim}

\begin{proof}
Suppose that $G$ has a vertex of degree $\Delta-1$. For each $u\in S$ with $d(u)\ge \Delta-k-1=n-2k-1$, Lemmas \ref{lemma1} and \ref{lemma4} give $|N_T(u)|=d(u)-s+1\ge d(u)-k+2$.
For sufficiently large $n$, by (\ref{eq2}) and Lemma \ref{ad}, we have
\begin{equation}\label{eq13}
\begin{aligned}
x_u&=\frac{1}{\rho}\sum_{w\sim u}x_w\ge\frac{m}{\rho}|N_T(u)|\ge\frac{m}{\rho}(d(u)-k+2)\ge\frac{M}{\alpha\rho}(n-3k+1)\\
&>\left(1-\frac{k}{\rho}\right)\frac{n-3k+1}{\rho}M>\frac{M}{2}.
\end{aligned}
\end{equation}
If $G$ has at least two vertices of degree $\Delta-1$, choose distinct such vertices $u,v$. By (\ref{eq13}), $x_u+x_v>M$, so Claim \ref{claim2} forces $d(G)=(\Delta^{n-2},(\Delta-1)^2)$.

It remains to suppose that $u$ is the unique vertex of degree $\Delta-1$. If $S'\setminus\{u\}\neq\emptyset$, choose $v\in S'\setminus\{u\}$.
Claim \ref{claim1} gives $d(v)\ge d(u)-k=\Delta-k-1$, so (\ref{eq13}) gives $x_u+x_v>M$.
On the other hand, by Claim \ref{claim2} we have $x_u+x_v\le M$, a contradiction.
Consequently, $S'=\{u\}$, and  vertex of $S\setminus\{u\}$ then has degree $s-1$. Therefore, $d(G)=(\Delta^{n-s},\Delta-1,(s-1)^{s-1})$.
\end{proof}

\begin{claim}\label{claim4}
If $d(G)\not\in \{(\Delta^{n-2},(\Delta-1)^2),(\Delta^{n-1},\Delta-2)\}\cup\{(\Delta^{n-s},\Delta-1,(s-1)^{s-1}): 1\le s\le k-1\}$, then $x_u\le M/2$ for all $u\in S$.
\end{claim}

\begin{proof}
Suppose that $d(G)$ is not one of the  exceptional degree sequences in the claim. By Claim \ref{claim3}, each vertex in $S$ has degree at most $\Delta-2$.
By Lemma \ref{lemma1}, we only need to prove $x_u\le M/2$ for $d(u)=\max\{d(w):w\in S\}$.
Let $T_1=T\setminus N(u)$. By Lemma \ref{lemma1}, we have $|T_1|\ge 2$.

If $G[T_1]$ is an empty graph, then each vertex $a\in T_1$ has no neighbor in $S$, since otherwise Lemma \ref{lemma1} would imply $a\in N_T(u)$, which contradicts $a\in T_1$. Thus all $\Delta$ neighbors of $a$ would lie in $N_T(u)$, but $|N_T(u)|=d(u)-s+1<\Delta$, a contradiction. Therefore,
$G[T_1]$ is nonempty and we can select $a,b\in T_1$ such that $a\nsim u$, $b\nsim u$, and $a\sim b$.

Consider \textbf{Operation~3}: $\boldsymbol{G'=G+ua+ub-ab}$; see Figure \ref{fig:3}. It is clear that $G'$ is connected.
In $G'$, the vertices $a,b$ retain degree $\Delta$, while $u$ gains two edges; hence $\Delta(G')=\Delta$.
Since $d(G)\ne(\Delta^{n-1},\Delta-2)$, $G'$ is nonregular.
Therefore,  $G'\in\mathcal{C}(n,\Delta)$, and the extremality of $G$ gives $\rho(G')-\rho(G)\le0$.
Moreover, we have
\begin{equation}\label{eq14}
\begin{aligned}
\frac{\rho(G')-\rho(G)}{2}&\ge x_u(x_a+x_b)-x_ax_b=x_u^2-(x_a-x_u)(x_b-x_u)\\
&\ge x_u^2-(M-x_u)^2=2Mx_u-M^2,
\end{aligned}
\end{equation}
which leads to $x_u\le M/2$.
\end{proof}
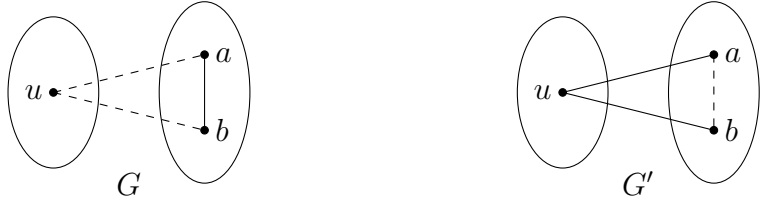
\begin{figure}[H]
	\centering
    \begin{subfigure}{.4\textwidth}
	\centering
	\begin{tikzpicture}
	    \draw  (-1,0) ellipse [x radius=0.6cm, y radius=1cm];
		\draw  (1,0) ellipse [x radius=0.6cm, y radius=1.2cm];
        \draw(-1,0)[fill] circle[radius=0.5mm];
        \draw(1,0.5)[fill] circle[radius=0.5mm];
        \draw(1,-0.5)[fill] circle[radius=0.5mm];
        \draw [dashed](-1,0)--(1,0.5);\draw [dashed](-1,0)--(1,-0.5);\draw(1,0.5)--(1,-0.5);
        \node at(-1,0)[left]{$u$};
        \node at(1,0.5)[right]{$a$};
        \node at(1,-0.5)[right]{$b$};
		\node at(0,-1.5)[above]{$G$};
	\end{tikzpicture}
    \end{subfigure}
    \begin{subfigure}{.4\textwidth}
	\centering
	\begin{tikzpicture}
	    \draw  (-1,0) ellipse [x radius=0.6cm, y radius=1cm];
		\draw  (1,0) ellipse [x radius=0.6cm, y radius=1.2cm];
        \draw (-1,0)[fill] circle[radius=0.5mm];
        \draw (1,0.5)[fill] circle[radius=0.5mm];
        \draw (1,-0.5)[fill] circle[radius=0.5mm];
        \draw (-1,0)--(1,0.5);\draw(-1,0)--(1,-0.5);\draw[dashed](1,0.5)--(1,-0.5);
        \node at(-1,0)[left]{$u$};
        \node at(1,0.5)[right]{$a$};
        \node at(1,-0.5)[right]{$b$};
		\node at(0,-1.5)[above]{$G'$};
	\end{tikzpicture}
    \end{subfigure}
	\caption{Operation 3}
	\label{fig:3}
\end{figure}

\begin{claim}\label{claim5}
If $d(G)\not\in\{ (\Delta^{n-s},s^2,(s-1)^{s-2}): 2\le s\le k-1\}$, then $x_u+x_v\ge m-(M-m)^2/m$  for all distinct vertices $u,v\in S'$.
\end{claim}

\begin{proof}
If $|S'|\le1$, the claim is vacuous. Suppose that $d(G)\neq(\Delta^{n-s},s^2,(s-1)^{s-2})$, and choose distinct $u,v\in S'$ such that $d(v)\le d(u)\le d(w)$ for all $w\in S'\setminus\{u,v\}$.
By Lemma  \ref{lemma1}, it suffices to prove the inequality for this pair. Since $v\in S'$, we have $N_T(v)\neq\emptyset$, and Lemma  \ref{lemma1} gives $N_T(v)\subseteq N_T(u)$.
Choose $a\in N_T(v)$.

Suppose first that $G[  T\setminus N(a)]$ contains an edge, and choose adjacent vertices $b,c\in  T\setminus N(a)$.
We consider \textbf{Operation~4}: $\boldsymbol{G'=G-ua-va+ab+ac-bc}$; see Figure \ref{fig:4}.
In $G'$, the vertices $a,b,c$ retain degree $\Delta$, while each of $u$ and $v$ loses one edge; hence $\Delta(G')=\Delta$.
Its deficient vertices are precisely the vertices of $S$, and since no edge inside $S$ is changed, they induce a complete graph of order at most $k-1$.
If the operation removed every edge between $S$ and $T$, then $S'=\{u,v\}$ and $d(u)=d(v)=s$, which contradicts $d(G)\neq(\Delta^{n-s},s^2,(s-1)^{s-2})$.
Thus $G'$ contains an edge between $S$ and $T$. Applying Lemma \ref{connect}, $G'$ is connected. Hence $G'\in\mathcal{C}(n,\Delta)$, and the extremality of $G$ gives $\rho(G')-\rho(G)\le0$.
Since $a,b,c\in T$ and $u,v\in S$, by Lemma \ref{lemma3} we have $x_b>x_u$ and $x_c>x_v$. Recall that $m\le x_a,x_b,x_c\le M$.
By Lemma \ref{tv}, we obtain
\begin{equation}\label{eq16}
\begin{aligned}
\frac{\rho(G')-\rho(G)}{2}&\ge x_ax_b+x_ax_c-x_ax_u-x_ax_v-x_bx_c\\
&=x_a(x_b+x_c-x_u-x_v)-x_bx_c\\
&\ge m(x_b+x_c-x_u-x_v)-x_bx_c\\
&=m^2-m(x_u+x_v)-(x_b-m)(x_c-m)\\
&\ge m^2-m(x_u+x_v)-(M-m)^2,
\end{aligned}
\end{equation}
which leads to $x_u+x_v\ge m-(M-m)^2/m$.

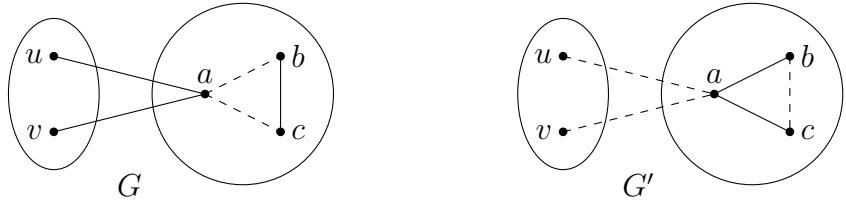
\begin{figure}[H]
	\centering
    \begin{subfigure}{.4\textwidth}
	\centering
	\begin{tikzpicture}
	    \draw  (-1,0) ellipse [x radius=0.6cm, y radius=1cm];
		\draw  (1.5,0) ellipse [x radius=1.2cm, y radius=1.2cm];
        \draw(-1,0.5)[fill] circle[radius=0.5mm];
        \draw(-1,-0.5)[fill] circle[radius=0.5mm];
        \draw(2,0.5)[fill] circle[radius=0.5mm];
        \draw(2,-0.5)[fill] circle[radius=0.5mm];
        \draw(1,0)[fill] circle[radius=0.5mm];
        \draw (-1,0.5)--(1,0);\draw (-1,-0.5)--(1,0);\draw (2,0.5)--(2,-0.5);
        \draw [dashed](1,0)--(2,0.5);\draw[dashed](1,0)--(2,-0.5);
        \node at(-1,0.5)[left]{$u$};
        \node at(-1,-0.5)[left]{$v$};
        \node at(2,0.5)[right]{$b$};
        \node at(1,0)[above]{$a$};
        \node at(2,-0.5)[right]{$c$};
		\node at(0,-1.5)[above]{$G$};
	\end{tikzpicture}
    \end{subfigure}
    \begin{subfigure}{.4\textwidth}
	\centering
	\begin{tikzpicture}
	    \draw  (-1,0) ellipse [x radius=0.6cm, y radius=1cm];
		\draw  (1.5,0) ellipse [x radius=1.2cm, y radius=1.2cm];
        \draw(-1,0.5)[fill] circle[radius=0.5mm];
        \draw(-1,-0.5)[fill] circle[radius=0.5mm];
        \draw(2,0.5)[fill] circle[radius=0.5mm];
        \draw(2,-0.5)[fill] circle[radius=0.5mm];
        \draw(1,0)[fill] circle[radius=0.5mm];
        \draw [dashed](-1,0.5)--(1,0);\draw [dashed](-1,-0.5)--(1,0);\draw [dashed](2,0.5)--(2,-0.5);
        \draw (1,0)--(2,0.5);\draw(1,0)--(2,-0.5);
        \node at(-1,0.5)[left]{$u$};
        \node at(-1,-0.5)[left]{$v$};
        \node at(2,0.5)[right]{$b$};
        \node at(1,0)[above]{$a$};
        \node at(2,-0.5)[right]{$c$};
		\node at(0,-1.5)[above]{$G'$};
	\end{tikzpicture}
    \end{subfigure}
	\caption{Operation 4}
		\label{fig:4}
\end{figure}

It remains to suppose that  $G[  T\setminus N(a)] $ is edgeless. For each $z\in S'$, by the choice of $v$ we have $d(z)\ge d(v)$. Applying Lemma  \ref{lemma1} we get $a\in N_T(v)\subseteq N_T(z)$, which implies $z\in N(a)$ for all $z\in S'$.
Hence $ S\setminus N(a)\subseteq S\setminus S'$. Since vertices in $S\setminus S'$ have no neighbor in $T$, it follows that
\begin{equation}\label{eqh19}
 N_{V\setminus N(a)}(z)=\emptyset \quad\text{for all}\quad  z\in  T\setminus N(a).
\end{equation}
Recall that $u,v\in N(a)$. At most $s-2\le k-3$ of the $k-1$ nonneighbors of $a$ lie in $S$, so we can choose $b\in(T\setminus N(a))\setminus\{a\}$.
Since $d(b)=\Delta=|N(a)|$, \eqref{eqh19} implies   $N(b)=N(a)$.

Consider \textbf{Operation~5}: $\boldsymbol{G'=G-va-vb+ab}$; see Figure \ref{fig:5}. In $G'$, the vertices $a,b$ retain degree $\Delta$, while $v$ loses two edges; hence $\Delta(G')=\Delta$.
Its deficient vertices are precisely the vertices of $S$, and since no edge inside $S$ is changed, they induce a complete graph of order at most $k-1$.
The edges $ua$ and $ub$ remain in $G'$, so Lemma \ref{connect} shows that $G'$ is connected. Hence $G'\in\mathcal{C}(n,\Delta)$, and the extremality of $G$ gives $\rho(G')-\rho(G)\le0$.

Since $a,b\in T$ and $v\in S$, Lemma \ref{lemma3} gives $x_a,x_b\ge m>x_v$.
By Lemma \ref{tv}, we have
\begin{equation}\label{eq17}
\begin{aligned}
0\ge\frac{\rho(G')-\rho(G)}{2}&\ge x_ax_b-x_ax_v-x_bx_v\\
&=x_a(x_b-x_v)-x_bx_v\ge m(x_b-x_v)-x_bx_v\\
&=-mx_v+x_b(m-x_v)\ge m^2-2mx_v,
\end{aligned}
\end{equation}
which leads to $x_v\ge m/2$. Since $d(u)\ge d(v)$, applying Lemma  \ref{lemma1} we have $x_u\ge x_v$, and hence $x_u+x_v\ge m$.
\end{proof}
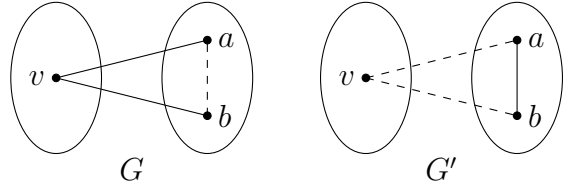
\begin{figure}[H]
	\centering
    \begin{subfigure}{.24\textwidth}
	\centering
	\begin{tikzpicture}
	    \draw  (-1,0) ellipse [x radius=0.6cm, y radius=1cm];
		\draw  (1,0) ellipse [x radius=0.6cm, y radius=1cm];
        \draw(-1,0)[fill] circle[radius=0.5mm];
        \draw(1,0.5)[fill] circle[radius=0.5mm];
        \draw(1,-0.5)[fill] circle[radius=0.5mm];
        \draw (-1,0)--(1,0.5);\draw (-1,-0)--(1,-0.5);\draw [dashed](1,-0.5)--(1,0.5);
        \node at(-1,0)[left]{$v$};
        \node at(1,0.5)[right]{$a$};
        \node at(1,-0.5)[right]{$b$};
		\node at(0,-1.5)[above]{$G$};
	\end{tikzpicture}
    \end{subfigure}
    \begin{subfigure}{.24\textwidth}
	\centering
	\begin{tikzpicture}
	    \draw  (-1,0) ellipse [x radius=0.6cm, y radius=1cm];
		\draw  (1,0) ellipse [x radius=0.6cm, y radius=1cm];
        \draw(-1,0)[fill] circle[radius=0.5mm];
        \draw(1,0.5)[fill] circle[radius=0.5mm];
        \draw(1,-0.5)[fill] circle[radius=0.5mm];
        \draw [dashed](-1,0)--(1,0.5);\draw [dashed](-1,0)--(1,-0.5);\draw (1,-0.5)--(1,0.5);
        \node at(-1,0)[left]{$v$};
        \node at(1,0.5)[right]{$a$};
        \node at(1,-0.5)[right]{$b$};
		\node at(0,-1.5)[above]{$G'$};
	\end{tikzpicture}
    \end{subfigure}
	\caption{Operation 5}
		\label{fig:5}
\end{figure}

\begin{claim}\label{claim6}
If $G$ has a vertex of degree $s$, then $d(G)=(\Delta^{n-s},s^2,(s-1)^{s-2})$ or $d(G)=(\Delta^{n-s},s,(s-1)^{s-1})$, where $2\le s\le k-1$ in the first sequence and $1\le s\le k-1$ in the second.
\end{claim}

\begin{proof}
Suppose that $G$ has a vertex of degree $s$. For each $u\in S$ with $d(u)\le s+k$, when $n$ is sufficiently large, by Lemma \ref{lemma3} and (\ref{eq2}), we have
\begin{equation}\label{eq18}
x_u=\frac{1}{\rho}\sum_{w\sim u}x_w\le\frac{M}{\rho}d(u)=\frac{\alpha m}{\rho}d(u)<\left(1+\frac{k}{\rho}\right)\frac{m}{\rho}(2k-1)<\frac{m}{4}.
\end{equation}
On the other hand, (\ref{eq1}) gives
\begin{equation}\label{eqh22}
m-(M-m)^2/m>(1-k^2/\rho^2)m>m/2
\end{equation} for sufficiently large $n$.

If $G$ has at least two vertices of degree $s$, choose distinct such vertices $u,v$. Then (\ref{eq18}) gives $x_u+x_v<m/2$, so Claim \ref{claim5} forces $d(G)=(\Delta^{n-s},s^2,(s-1)^{s-2})$.

It remains to suppose that $G$ has a unique vertex $u$ of degree $s$. If $S'\setminus\{u\}\neq\emptyset$, choose $v\in S'\setminus\{u\}$.
Claim \ref{claim1} gives $d(v)\le d(u)+k=s+k$, so (\ref{eq18}) gives $x_u+x_v<m/2$.
By Claim \ref{claim5} and \eqref{eqh22}, we get $x_u+x_v>m/2$, a contradiction.
Consequently, $S'=\{u\}$. It follows that each vertex of $S\setminus\{u\}$   has degree $s-1$, and hence $d(G)=(\Delta^{n-s},s,(s-1)^{s-1})$.
\end{proof}

\begin{claim}\label{claim7}
If
\begin{eqnarray*}
d(G)&\not\in&\left\{(\Delta^{n-s},s^2,(s-1)^{s-2}): 2\le s\le k-1\right\}\\
&&\cup\left\{ (\Delta^{n-s},s,(s-1)^{s-1}),(\Delta^{n-s},s+1,(s-1)^{s-1}): 1\le s\le k-1\right\},
\end{eqnarray*}
  then
$x_v\ge m/2-(M-m)^2/(2m)$
for all $v\in S'$.
\end{claim}

\begin{proof}
Suppose that $d(G)$ does not belong to the exceptional sequences in the claim. If $S'=\emptyset$, the claim is vacuous. So we assume that $S'\neq\emptyset$.
By Claim \ref{claim6}, each vertex in $S'$ has degree at least $s+1$.
Choose $v\in S'$ with $d(v)=\min\{d(w):w\in S'\}$. By Lemma  \ref{lemma1}, it suffices to prove the assertion for this vertex.
Put $W=N_T(v)$. Since $d(v)\ge s+1$, we have $|W|=d(v)-s+1\ge2$.

Suppose first that $G[W]$ is not complete, and choose distinct nonadjacent vertices $a,b\in W$.
Consider  Operation~5  again. We have $\Delta(G')=\Delta$ and its
  deficient vertices are precisely the vertices of $S$, which induce a complete graph of order at most $k-1$.
An edge between $S$ and $T$ remains unless $S'=\{v\}$ and $d(v)=s+1$, which is the excluded degree sequence.
Hence Lemma \ref{connect} shows that $G'$ is connected, and therefore $G'\in\mathcal{C}(n,\Delta)$.
Using the same argument as in the proof of Claim 5, we have \eqref{eq17}
and hence $x_v\ge m/2$.

Suppose now that $G[W]$ is complete, and choose distinct $a,b\in W$. Since $v$ has minimum degree in $S'$, by Lemma  \ref{lemma1} we have
\begin{equation}\label{eqh23}
N_T(v)\subseteq N_T(u) \quad\text{for all}\quad  u\in S'.
\end{equation}
Thus $a$ is adjacent to all vertices of $S'$ and to no vertex of $S\setminus S'$. Since $S'\neq\emptyset$ and $s\le k-1$, at most $s-1\le k-2$ of the $k-1$ nonneighbors of $a$ lie in $S$.
Then $T$ contains at least one nonneighbor of $a$, say $c$. Denote $Y=V\setminus N[b]$. Then $|Y|=k-1$. We assert that
there exists $d\in Y$ such that $c\sim d$. Otherwise, if $c\notin Y$, then the $k$ vertices in $Y\cup\{a\}$ are nonneighbors of $c$; if $c\in Y$, then the $k$ vertices in $(Y\setminus\{c\})\cup\{a,b\}$ are nonneighbors of $c$. Both alternatives contradict $d(c)=\Delta$.

\begin{samepage}
  If $d\in S'$, then by \eqref{eqh23} we have $b\in N_T(v)\subseteq N_T(d)$, which contradicts $d\in Y$.
If $d\in S\setminus S'$, then $d$ has no neighbor in $T$, which contradicts $c\sim d$. Hence $d\in T$.
\end{samepage}
The vertices $a,b,c,d$ are distinct, with $ac,bd\notin E(G)$ and $va,vb,cd\in E(G)$. Now we consider \textbf{Operation~6}: $\boldsymbol{G'=G-va-vb+ac+bd-cd}$; see Figure \ref{fig:6}.

\begin{figure}[H]
	\centering
    \begin{subfigure}{.4\textwidth}
	\centering
	\begin{tikzpicture}
	    \draw  (-1,0) ellipse [x radius=0.6cm, y radius=1cm];
		\draw  (1.5,0) ellipse [x radius=1.2cm, y radius=1.2cm];
        \draw(-1,0)[fill] circle[radius=0.5mm];
        \draw(1,0.5)[fill] circle[radius=0.5mm];
        \draw(1,-0.5)[fill] circle[radius=0.5mm];
        \draw(2,0.5)[fill] circle[radius=0.5mm];
        \draw(2,-0.5)[fill] circle[radius=0.5mm];
        \draw (-1,0)--(1,0.5);\draw(-1,0)--(1,-0.5);
        \draw [dashed](1,0.5)--(2,0.5);\draw[dashed](1,-0.5)--(2,-0.5);\draw(2,-0.5)--(2,0.5);
        \node at(-1,0)[left]{$v$};
        \node at(1,0.5)[above]{$a$};
        \node at(1,-0.5)[below]{$b$};
        \node at(2,0.5)[above]{$c$};
        \node at(2,-0.5)[below]{$d$};
		\node at(0,-1.5)[above]{$G$};
	\end{tikzpicture}
    \end{subfigure}
    \begin{subfigure}{.4\textwidth}
	\centering
	\begin{tikzpicture}
	     \draw  (-1,0) ellipse [x radius=0.6cm, y radius=1cm];
		\draw  (1.5,0) ellipse [x radius=1.2cm, y radius=1.2cm];
        \draw(-1,0)[fill] circle[radius=0.5mm];
        \draw(1,0.5)[fill] circle[radius=0.5mm];
        \draw(1,-0.5)[fill] circle[radius=0.5mm];
        \draw(2,0.5)[fill] circle[radius=0.5mm];
        \draw(2,-0.5)[fill] circle[radius=0.5mm];
        \draw [dashed](-1,0)--(1,0.5);\draw[dashed](-1,0)--(1,-0.5);
        \draw (1,0.5)--(2,0.5);\draw(1,-0.5)--(2,-0.5);\draw[dashed](2,-0.5)--(2,0.5);
        \node at(-1,0)[left]{$v$};
        \node at(1,0.5)[above]{$a$};
        \node at(1,-0.5)[below]{$b$};
        \node at(2,0.5)[above]{$c$};
        \node at(2,-0.5)[below]{$d$};
		\node at(0,-1.5)[above]{$G'$};
	\end{tikzpicture}
    \end{subfigure}
	\caption{Operation 6}
		\label{fig:6}
\end{figure}
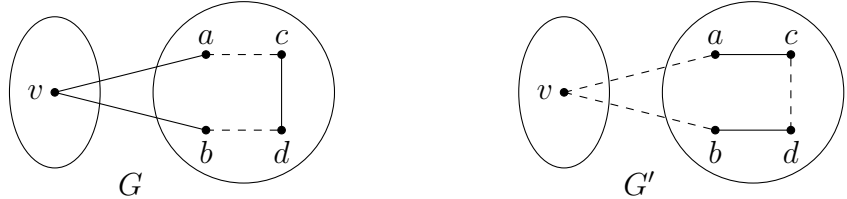

In $G'$, the vertices $a,b,c,d$ retain degree $\Delta$, while $v$ loses two edges; hence $\Delta(G')=\Delta$.
Its deficient vertices are precisely the vertices of $S$, and since no edge inside $S$ is changed, they induce a complete graph of order at most $k-1$.
As in the first case,  $G'$ is connected, and therefore $G'\in\mathcal{C}(n,\Delta)$.
Since $a,b,c,d\in T$ and $v\in S$, by Lemma \ref{lemma3} we have $x_c,x_d>x_v$ and $m\le x_a,x_b,x_c,x_d\le M$.
By he extremality of $G$ and Lemma \ref{tv}, we obtain
$$
\begin{aligned}
0&\ge x_a(x_c-x_v)+x_b(x_d-x_v)-x_cx_d\\
&\ge m(x_c+x_d-2x_v)-x_cx_d\\
&=m^2-2mx_v-(x_c-m)(x_d-m)\\
&\ge m^2-2mx_v-(M-m)^2,
\end{aligned}
$$
which leads to $x_v\ge m/2-(M-m)^2/(2m)$.
\end{proof}
\begin{claim}\label{claim8}
If $d(G)\notin\mathcal{D}$, then
$$
\frac{m}{2}-\frac{(M-m)^2}{2m}\le x_u\le\frac{M}{2}\qquad\text{and}\qquad\frac{n}{2}-\frac{3k-1}{2}\le d(u)\le\frac{n}{2}+\frac{3k-1}{2}
$$ for all $u\in S'$.
\end{claim}

\begin{proof}
Assume that $d(G)\notin\mathcal{D}$, and let $u\in S'$. By Claim  \ref{claim4} and Claim \ref{claim7} we have
\begin{equation}\label{eqh24}
\frac{m}{2}-\frac{(M-m)^2}{2m}\le x_u\le\frac{M}{2}.
\end{equation}
Consider the  Perron equation at $u$
$$\rho x_u=\sum_{w\in N_T(u)}x_w+\sum_{v\in S\setminus \{u\}}x_v.$$
By Lemma \ref{lemma3}, we get
\begin{equation}\label{eq19}
 (d(u)-s+1)m\le\rho x_u\le(d(u)-s+1)M+(s-1)m\le d(u)M.
\end{equation}
Together with \eqref{eqh24}, (\ref{eq19}) yields
$$
\frac{\rho}{2}(2-\alpha)=\frac{\rho}{\alpha}\left(\frac12-\frac{(\alpha-1)^2}{2}\right)\le \frac{\rho x_u}{M}\le d(u)\le\frac{\rho x_u}{m}+s-1\le\frac{\rho\alpha}{2}+s-1.
$$
By (\ref{eq2}), we have $2-\alpha>1-k/\rho$ and $\rho\alpha<\rho+k$. Since $n-2k+1<\rho\le n-k$ and $s\le k-1$, we obtain
$$
\begin{aligned}
d(u)&\ge\frac{\rho}{2}(2-\alpha)>\frac{\rho-k}{2}>\frac{n}{2}-\frac{3k-1}{2},\\
d(u)&\le\frac{\rho\alpha}{2}+s-1<\frac{\rho+k}{2}+s-1\le\frac{n}{2}+s-1\le\frac{n}{2}+k-2<\frac{n}{2}+\frac{3k-1}{2}.
\end{aligned}
$$
\end{proof}

Since $G$ is connected and $G[S]$ is complete, we have $S'\ne\emptyset$.
Suppose that $d(G)\notin\mathcal{D}$, and choose $u\in S'$ such that $x_u=\max_{v\in S'p}x_v$. By Claim \ref{claim8}, we have
\begin{equation}\label{eqh26}
\frac{m}{2}-\frac{(M-m)^2}{2m}\le x_u\le\frac{M}{2}.
\end{equation}
By the choice of $u$, Lemma  \ref{lemma1} imply that each vertex of $T\setminus N_T(u)$ has no neighbor in $S$. Starting from $G$, we repeatedly apply  {Operation~3} at $u$ until the degree of $u$ reaches the member of $\{\Delta-2,\Delta-1\}$ having the same parity as $d(u)$. This is possible at every step. In fact, the operations only add edges from $u$ to $T$ and delete edges inside $T$, so each vertex remaining in $W:=T\setminus N_T(u)$ has no neighbor in $S$. Since $u$ remains deficient, we have $|N_T(u)|<\Delta<|T|$, and hence $W\ne\emptyset$. If $W$ were independent, each of its vertices would have all its $\Delta$ neighbors in $N_T(u)$, contradicting $|N_T(u)|<\Delta$. Therefore, we can always find two adjacent vertices in $W$ to repeat Operation 3.

Let $r$ be the number of operations performed and $G'$ the resulting graph. Notice that each operation preserves the degree $\Delta$ of all vertices in $T$, increases the degree of $u$ by two while keeping it below $\Delta$, changes no other vertex degree, and changes no edge of $G[S]$. Moreover, $\partial(S,T)$ remains nonempty. Hence Lemma \ref{connect} shows that every intermediate graph, and in particular $G'$, belongs to $\mathcal{C}(n,\Delta)$. Since the final degree of $u$ belongs to $\{\Delta-2,\Delta-1\}$, by Claim \ref{claim8} we have
$$
\frac{n}{4}-\frac{5k+3}{4}\le r\le\frac{n}{4}+\frac{k-3}{4}.
$$

Let $Z$ be the vector obtained from $X$ by replacing $x_u$ with $m$, and let $a_ib_i$ be the edge removed in the $i$-th operation. Denote the adjacency matrices of $G$ and $G'$ by $A$ and $A'$, respectively. Set
$$
 \qquad \varepsilon=\frac{M-m}{m}\qquad\text{and}\qquad\beta=\frac{k}{n-2k+1}.
$$
Recall that $\rho> n-2k+1$. By (\ref{eq1}), we have $0\le\varepsilon<k/\rho<\beta$. If $n>5k+3$, then $\beta<1$ and $n/2-(5k+3)/2>0$; hence by \eqref{eqh26} we have $x_u\le M/2=(1+\varepsilon)m/2<m$. Claim \ref{claim8} also gives
$$
0<m-x_u\le\frac{m}{2}(1+\varepsilon^2).
$$
Since $AX=\rho X$ and
$$
m(x_{a_i}+x_{b_i})-x_{a_i}x_{b_i}=m^2-(x_{a_i}-m)(x_{b_i}-m)\ge m^2-(M-m)^2,
$$
we have
\begin{equation}\label{eq20}
\begin{aligned}
Z^TA'Z-\rho Z^TZ&=(Z^TAZ-\rho Z^TZ)+Z^T(A'-A)Z\\
&=-\rho(m-x_u)^2+2\sum_{i=1}^r\big(m(x_{a_i}+x_{b_i})-x_{a_i}x_{b_i}\big)\\
&\ge-\rho(m-x_u)^2+2r\big(m^2-(M-m)^2\big)\\
&\ge\left(2r(1-\varepsilon^2)-\frac{\rho}{4}(1+\varepsilon^2)^2\right)m^2\\
&>\left(\left(\frac{n}{2}-\frac{5k+3}{2}\right)(1-\beta^2)-\frac{n-k}{4}(1+\beta^2)^2\right)m^2.
\end{aligned}
\end{equation}
The final coefficient in (\ref{eq20}) is positive whenever
$$
2(n-5k-3)(1-\beta^2)>(n-k)(1+\beta^2)^2.
$$
Since this explicit inequality holds for all sufficiently large $n$, it follows from (\ref{eq20}) and the Rayleigh quotientp that $\rho(G')>\rho(G)$, contradicting $G\in\mathcal{G}(n,\Delta)$. Therefore,  we obtain  $d(G)\in \mathcal{D}$.
\end{proof}

\section{Proof of Theorem 1}

Define $M_1(a,b)$ and $M_2(a,b)$  to be  nonnegative matrices  as
$$M_1(a,b)=
\begin{bmatrix}
k & a & b & 0 \\
1 & k+a-1 & 0 & 0 \\
1 & 0 & a+b & n-a-b-1 \\
0 & 0 & b & n-b
\end{bmatrix}$$
and
$$M_2(a,b)=
\begin{bmatrix}
k+1 & a & b & 0 \\
2 & k+a-1 & 0 & 0 \\
2 & 0 & a+b & n-a-b-2 \\
0 & 0 & b & n-b
\end{bmatrix},$$
where the parameters $a\ge 0, b\ge 1$ satisfies $a+b\le n-1$ in $M_1(a,b)$ and $a+b\le n-2$ in $M_2(a,b)$. Let $\lambda_{a,b}=\rho(M_1(a,b))$ and $\mu_{a,b}=\rho(M_2(a,b))$ be their largest real eigenvalues.

\begin{defi}
Let $G\in\mathcal{C}(n,n-k)$.
Let $A_G$ be the adjacency matrix of $G$, and define $B_G=A_G+kI_n$.
Then $B_G$ is a nonnegative irreducible matrix. Moreover, $B_G$ and $A_G$ have the same   Perron vector.
For a vertex $p\in V(G)$, define
$$\mathcal{L}(G,p)=\{(u,v)\in N_T(p)\times(T\setminus N_T(p)):uv\notin E(G)\}.$$
We define the \textbf{matrix operation} with respect to $G$ and $p$ by replacing $B_G$ with
$$B_{G,p}:=B_G-\sum_{(u,v)\in\mathcal{L}(G,p)}(e_u-e_v)(e_u-e_v)^T.$$

Notice that all entries of  $(e_u-e_v)(e_u-e_v)^T$ are zero except the  $2\times 2$ principal submatrix
$$\begin{bmatrix}
1&-1\\
-1&1
\end{bmatrix},$$
which lies in rows and columns corresponding to the vertices $u,v$.
Since each vertex in $T$ is incident to at most $k-1$ pairs in $\mathcal{L}(G,p)$,   $B_{G,p}$ is nonnegative and irreducible. Moreover, the matrix
  $B_G-B_{G,p}$ is positive semidefinite.
Let $X^*=(x_u^*)_{u\in V(G)}$ be the  Perron vector of $B_{G,p}$, where $x_u^*$ denotes the component corresponding to the vertex $u$.
\end{defi}

Assume that $G\in \mathcal{G}(n,\Delta)$. Define
$$
\begin{aligned}
\mathcal{D}_1:=&\left\{(\Delta^{n-1},\Delta-2)\right\}\cup\\
&\{(\Delta^{n-s},\Delta-1,(s-1)^{s-1}),(\Delta^{n-s},s+1,(s-1)^{s-1}),(\Delta^{n-s},s,(s-1)^{s-1}):1\le s<k\}
\end{aligned}
$$
  and
$$\mathcal{D}_2=\{(\Delta^{n-2},(\Delta-1)^2)\}\cup\{(\Delta^{n-s},s^2,(s-1)^{s-2}):2\le s\le k-1\}.$$
Then $\mathcal{D}=\mathcal{D}_1\cup\mathcal{D}_2$, and Lemma \ref{lemma5} gives $d(G)\in\mathcal{D}_1\cup\mathcal{D}_2$. If $d(G)\in\mathcal{D}_i$, then $s_1=i$ for $i=1,2$. If $s_1=2$, write $S'=\{p_1,p_2\}$. Then applying Lemma \ref{lemma1} we have $N_T(p_1)=N_T(p_2)$.

\begin{claim}\label{claim9}
For each $p\in S'$, we have
$$
0\le \rho(G)+k-\rho(B_{G,p})<\frac{k^2(k^2-1)}{(n-s)\rho(G)^2}.
$$
\end{claim}

\begin{proof}
Fix $p\in S'$. Let $X=(x_v)_{v\in V(G)}$ be the   Perron vector of $G$, and set
$$
\rho=\rho(G),\qquad M=\max_{u\in T}x_u,\qquad m=\min_{u\in T}x_u.
$$
It follows from the definitions of $\mathcal{D}_1$ and $\mathcal{D}_2$ that either $|N_T(p)|\in\{1,2\}$ or $|T\setminus N_T(p)|\in\{k,k+1\}$.
Since each $u\in T$ has exactly $k-1$ nonneighbors in $V(G)\setminus\{u\}$, counting the pairs in $\mathcal{L}(G,p)$ from the smaller of the two vertex classes gives
\begin{equation}\label{eqhh28}p
|\mathcal{L}(G,p)|\le(k-1)\min\{|N_T(p)|,|T\setminus N_T(p)|\}\le(k-1)(k+1)=k^2-1.
\end{equation}

By Lemma  \ref{lemma1} and Lemma \ref{lemma6} we have \eqref{eq1}.
Recall that $B_G-B_{G,p}$ is positive semidefinite and $\rho(B_G)=\rho+k$ with
  $X$  aslo the  Perron vector of $B_G$.  By  Lemma \ref{tv} and  (\ref{eq1}), we obtain
$$
\begin{aligned}
0\le \rho+k-\rho(B_{G,p})
&\le\sum_{(u,v)\in\mathcal{L}(G,p)}(x_u-x_v)^2 \le|\mathcal{L}(G,p)|(M-m)^2
<\frac{k^2(k^2-1)}{(n-s)\rho^2}.
\end{aligned}
$$
\end{proof}

\begin{claim}\label{claim10}
Let $p\in S'$, and set $a=s_2$ and $b=|N_T(p)|$. If $d(G)\in\mathcal{D}_1$, then $\rho(B_{G,p})=\lambda_{a,b}$; if $d(G)\in\mathcal{D}_2$, then $\rho(B_{G,p})=\mu_{a,b}$.
\end{claim}

\begin{proof}
Let $d(G)\in\mathcal{D}_i$ with $i=s_1\in\{1,2\}$. By the definitions of $S'$ and $s_2$, the four sets
$$
S',\qquad S\setminus S',\qquad N_T(p),\qquad T\setminus N_T(p)
$$
have cardinalities $i,a,b,n-a-b-i$, respectively. By Lemma \ref{lemma1}, the graph $G[S]$ is complete and no vertex of $S\setminus S'$ has a neighbor in $T$. Moreover, if $i=2$, then the two vertices of $S'$ have the same neighborhood in $T$. By the definition of $B_{G,p}$, the ordered partition formed by these four sets is therefore equitable whenever $a>0$, whose quotient matrix is $M_i(a,b)$.

If $a=0$, then  the  set $S\setminus S'$ is empty and  the resulting quotient matrix is
$$
\widetilde M_i(b):=
\begin{bmatrix}
k+i-1 & b & 0\\
i & b & n-b-i\\
0 & b & n-b
\end{bmatrix},
$$
which is obtained from $M_i(0,b)$ by deleting its second row and column.  Notice that
$$
\begin{aligned}
\rho(\widetilde M_i(b))&\ge\rho\left(\begin{bmatrix}b&n-b-i\\b&n-b\end{bmatrix}\right) \ge\min\{n-i,n\}=n-i>k-1,
\end{aligned}
$$
for sufficiently large $n$. It follows  that $\rho(M_i(0,b))=\rho(\widetilde M_i(b))$.

Now applying Lemma \ref{QM}  we obtain $\rho(B_{G,p})=\rho(M_i(a,b))$, which is $\lambda_{a,b}$ for $i=1$ and $\mu_{a,b}$ for $i=2$.
\end{proof}

\begin{claim}\label{claim11}
If $n$ is sufficiently large and $b=O(1)$, then
\begin{align}
\lambda_{a,b}
&=n-\frac{b}{n}+\frac{b^2-ab}{n^2}+O\left(\frac{1}{n^3}\right),
 \label{eq26}\\
\lambda_{a,\Delta-b}
&=n-\frac{b}{n}+\frac{b^2-a(b-1)}{n^2}+O\left(\frac{1}{n^3}\right),
 \label{eq27}\\
\mu_{a,1}
&=n-\frac{2}{n}+\frac{2-2a}{n^2}+O\left(\frac{1}{n^3}\right).\label{eq28}
\end{align}
Moreover,
\begin{equation}\label{eqh31}
\lambda_{0,b}=\lambda_{0,\Delta-b},\qquad
\mu_{0,b}=\mu_{0,\Delta-1-b}
\end{equation}
whenever both sides are defined.
The constants in the error terms may depend on the fixed integer $k$.
\end{claim}

\begin{proof}
Let $f_i(a,b,\lambda)=\det(\lambda I-M_i(a,b))$. A direct calculation gives
\begin{align*}
f_1(0,b,\lambda)
&=(\lambda-k+1)\{\lambda(\lambda-n)(\lambda-k)+b(\Delta-b)\},\\
f_2(0,b,\lambda)
&=(\lambda-k+1)\{\lambda(\lambda-n)(\lambda-k-1)+2b(\Delta-1-b)\}.
\end{align*}
The  identities in \eqref{eqh31} are straightforward.

Put $q=\lambda-k-a+1$ and $D=\lambda-n+b$. A direct expansion gives
\begin{eqnarray}
f_1(a,b,\lambda)&=&((\lambda-k)q-a)((\lambda-n)(\lambda-a)+b)-bqD,\label{eqh32}\\
f_2(a,b,\lambda)&=&((\lambda-k-1)q-2a)((\lambda-n)(\lambda-a)+2b)-2bqD.\label{eqh33}
\end{eqnarray}
Setting  $b=\Delta-r$ in \eqref{eqh32} we get
\begin{align}\label{eqh34}
f_1(a,\Delta-r,\lambda)
=q\left[(\lambda-n)(\lambda-k)(\lambda-a)+r(\Delta-r)\right]
-a[(\lambda-n)(\lambda-a)+\Delta-r].
\end{align}
Let
$$
\widehat\lambda_1=n-\frac{b}{n}+\frac{b^2-ab}{n^2},\qquad
\widehat\lambda_2=n-\frac{b}{n}+\frac{b^2-a(b-1)}{n^2},\qquad
\widehat\mu=n-\frac{2}{n}+\frac{2-2a}{n^2}.
$$
Then  for bounded $a,b$, we have
\begin{align*}
f_1(a,b,\widehat\lambda_1)
&=(n^2+O(n))\left(\frac{b^2}{n}+O(n^{-2})\right)
-\left(b^2n+O(1)\right)=O(1),\\
f_1(a,\Delta-b,\widehat\lambda_2)
&=(n+O(1))(a+O(n^{-1}))-a(n+O(1))=O(1),\\
f_2(a,1,\widehat\mu)
&=(n^2+O(n))\left(\frac{2}{n}+O(n^{-2})\right)
-\left(2n+O(1)\right)=O(1).
\end{align*}
\indent For $i\in\{1,2\}$, the lower-right $2\times2$ principal submatrix of $M_i(a,b)$ and the row sums of $M_i(a,b)$ give
\begin{align*}
n-i
&\le \rho\left(
\begin{bmatrix}
a+b & n-a-b-i\\
b & n-b
\end{bmatrix}
\right)
\le \rho(M_i(a,b))\\
&\le \max\{k+i-1+a+b,k+a+i-1,n\}
\le n+k-1.
\end{align*}
Thus the three Perron roots and $\widehat\lambda_1$, $\widehat\lambda_2,$ $\widehat\mu$  all lie in a fixed interval $I_n=[n-C,n+C]$, where $C$ depends only on $k$.

For all sufficiently large $n$ and $\lambda\in I_n$, differentiating \eqref{eqh31} and \eqref{eqh32} with respect to $\lambda$
and noting that $\partial_\lambda q=\partial_\lambda D=1$, we have
\begin{align*}
\frac{\partial f_1}{\partial\lambda}
&=(q+\lambda-k)((\lambda-n)(\lambda-a)+b)+((\lambda-k)q-a)(2\lambda-n-a) -b(q+D)\\
&=n^3+O(n^2)\ge \frac12n^3,\\
\frac{\partial f_2}{\partial\lambda}
&=(q+\lambda-k-1)((\lambda-n)(\lambda-a)+2b)+((\lambda-k-1)q-2a)(2\lambda-n-a)-2b(q+D)\\
&=n^3+O(n^2)\ge \frac12n^3.
\end{align*}

Let $\xi_1,\xi_2,\xi_3$ lie between the relevant Perron root and $\widehat\lambda_1$, $\widehat\lambda_2$ and $\widehat\mu$, respectively. Applying the mean value theorem we obtain
\begin{align*}
|\lambda_{a,b}-\widehat\lambda_1|
&=\frac{|f_1(a,b,\widehat\lambda_1)|}
{|\partial_\lambda f_1(a,b,\xi_1)|}=O(n^{-3}),\\
|\lambda_{a,\Delta-r}-\widehat\lambda_2|
&=\frac{|f_1(a,\Delta-r,\widehat\lambda_2)|}
{|\partial_\lambda f_1(a,\Delta-r,\xi_2)|}=O(n^{-3}),\\
|\mu_{a,1}-\widehat\mu|
&=\frac{|f_2(a,1,\widehat\mu)|}
{|\partial_\lambda f_2(a,1,\xi_3)|}=O(n^{-3}).
\end{align*}
This proves (\ref{eq26})--(\ref{eq28}).
\end{proof}

Combining Claims \ref{claim10} and \ref{claim11}, we obtain the following estimates for the candidates in $\mathcal D$:
\begin{table}[H]
\centering
\caption{Asymptotic estimates for the quotient spectral radii of the candidates in $\mathcal D$.}
\label{tab:quotient-root-estimates}
\[
\begin{array}{c|c}
\text{degree sequence}&\rho(B_{G,p})\\ \hline
(\Delta^{n-s},\Delta-1,(s-1)^{s-1})
&n-\dfrac{s}{n}+\dfrac{2s-1}{n^2}+O(n^{-3})\\[3mm]
(\Delta^{n-1},\Delta-2)
&n-\dfrac{2}{n}+\dfrac{4}{n^2}+O(n^{-3})\\[3mm]
(\Delta^{n-s},s+1,(s-1)^{s-1})
&n-\dfrac{2}{n}+\dfrac{6-2s}{n^2}+O(n^{-3})\\[3mm]
(\Delta^{n-s},s,(s-1)^{s-1})
&n-\dfrac{1}{n}+\dfrac{2-s}{n^2}+O(n^{-3})\\[3mm]
(\Delta^{n-s},s^2,(s-1)^{s-2})
&n-\dfrac{2}{n}+\dfrac{6-2s}{n^2}+O(n^{-3})\\[3mm]
(\Delta^{n-2},(\Delta-1)^2)
&n-\dfrac{2}{n}+\dfrac{2}{n^2}+O(n^{-3})
\end{array}
\]
\end{table}
where the constants may depend on $k$.
\begin{claim}\label{claim12}
 Suppose $\delta\in \{1,2,\Delta-2,\Delta-1\}$. If $(n-1)\Delta+\delta$ is even, then there is a graph $G_\delta\in\mathcal{C}(n,\Delta)$ such that  $d(G_{\delta})=(\Delta^{n-1},\delta)$ and $\rho(G_{\delta})+k\ge \lambda_{0,\delta}$.
\end{claim}
\begin{proof}
For each admissible $\delta$, put
$$
Q_{\delta}:=
\begin{bmatrix}
k&\delta&0\\
1&\delta&n-\delta-1\\
0&\delta&n-\delta
\end{bmatrix}.
$$
This is the matrix $\widetilde M_1(\delta)$ in the proof of Claim \ref{claim10}, and the argument there gives $\rho(Q_{\delta})=\lambda_{0,\delta}$. We also use the fact that an $r$-regular graph of order $N$ exists whenever $0\le r\le N-1$ and $Nr$ is even.

First suppose that $\delta\in\{\Delta-2,\Delta-1\}$.  We assert that there exits a $(\delta-k)$-regular graph $H_1$ and a $(\Delta-\delta)$-regular graph $H_2$ such that $$|V(H_1)|=\delta \quad\text{and} \quad |V(H_2)|=n-\delta-1.$$ In fact,  If $\delta=\Delta-1$, the parity assumption makes $n\Delta-1$ even. Therefore, $n$ and $\Delta$ are both odd, which implies that $|V(H_1)|=\Delta-1$ and $|V(H_2)|=n-\Delta$ are both even. If $\delta=\Delta-2$, then $n\Delta$ is even. Then either $\Delta$ is even,  $|V(H_1)|$ is even, or
  $\Delta$ is odd, forcing   $n$   even, $k$   odd, and $\delta-k$  even. Moreover, $\Delta-\delta=2$.
  Thus $H_1$ and $H_2$ always exist.

  Denote $V_1=V(H_1)$ and $V_2=V(H_2)$. We construct a graph $G_\delta$ by adding a new vertex $u$ to $H_1\vee H_2$ and adding an edge between $u$ and every vertex of $V_1$.
Then
$$
d(u)=\delta,\qquad d(x)=1+(\delta-k)+(n-\delta-1)=\Delta,\qquad d(y)=\delta+(\Delta-\delta)=\Delta
$$
for $x\in V_1$ and $y\in V_2$. Hence $d(G_{\delta})=(\Delta^{n-1},\delta)$. Moreover, $G_\delta$ is connected, and $\{\{u\},V_1,V_2\}$ is an equitable partition for $B_{G_{\delta}}$, with quotient matrix $Q_{\delta}$. Therefore,
$$
\rho(G_{\delta})+k=\rho(Q_{\delta})=\lambda_{0,\delta}.
$$

Now suppose that $\delta\in\{1,2\}$. We assert that there exists a $(\Delta-k-\delta+1)$-regular graph $H_3$ and  a $\delta$-regular graph $H_4$ such that
$$|V(H_3)|=\Delta-\delta\quad \text{and}\quad |V(H_4)|=k-1.$$
In fact,
if $\delta=1$, the parity assumption implies that $n$ is even, and $\Delta$, $k$ are odd. Therefore, $|V(H_3)|=\Delta-1$ and $|V(H_4)|=k-1$ are even. If $\delta=2$, then $(n-1)\Delta$ is even. Therefore, either $\Delta$ and hence $|V(H_3)|$ are even, or $n$ is odd, $k$ is even, and the required degree $\Delta-k-1$ on $V(H_3)$ is even.  Hence, the required regular graph $H_3$ and $H_4$ always exist.

Denote  $V_3=V(H_3)$ and $V_4=V(H_4)$. For $\delta=1$, we construct $G_1$ by adding two new vertices $u,v$ and the edge $uv$ to $H_3\vee H_4$ and joining $v$ to every vertex of $V_3$; for $\delta=2$, we construct $G_2$ by adding three new vertices $u,v,w$ inducing a triangle to $H_3\vee H_4$ and joining both $v,w$ to every vertex of $V_3$.
  Then $G_\delta$ is connected with $d(G_{\delta})=(\Delta^{n-1},\delta)$. Taking the matrix operation with $p=u$, we have $\mathcal L(G_{\delta},u)=U\times V_4$, where $U=\{v\}$ for $\delta=1$ and $U=\{v,w\}$ for $\delta=2$. Hence, $\{\{u\},U,V_3\cup V_4\}$ is an equitable partition for $B_{G_{\delta},u}$, with quotient matrix $Q_{\delta}$. Since $B_{G_{\delta}}-B_{G_{\delta},u}$ is positive semidefinite, we have
\begin{equation*}
\rho(G_{\delta})+k=\rho(B_{G_{\delta}})\ge \rho(B_{G_{\delta},u})=\rho(Q_{\delta})=\lambda_{0,\delta}.\qedhere
\end{equation*}
\end{proof}

\begin{claim}\label{claim13}
Suppose that $n$ and $k$ have the same parity with $n$ sufficiently large. If
 $G$ is connected with $d(G)=(\Delta^{n-1},\Delta-2)$, then
$$
\rho(G)+k\le \lambda_{0,\Delta-2}+O(n^{-5}).
$$
Moreover, the graph $G_2$ constructed in Claim \ref{claim12} satisfies
$$
\rho(G_2)+k\ge \lambda_{0,2}+\frac{k-1}{n^3}.
$$
\end{claim}

\begin{proof}
By Claim \ref{claim11}, $\lambda_{0,\Delta-2}=\lambda_{0,2}$; denote this common value by $\lambda_0$. Let $p$ be the vertex of degree $\Delta-2$ in $G$.
Put $T=V(G)\setminus\{p\}$, $U=N_T(p)$, and $W=T\setminus U$.
Thus $|U|=\Delta-2$ and $|W|=k+1$.

By the definition of the matrix operation, each entry of $B_{G,p}$ between $U$ and $W$ equals one. Denote by $B_{G,p}[R]$ the principal submatrix of $B_{G,p}$ with rows and columns indexed by $R$.
If $v\in U$ has $\ell_v$ nonneighbors in $W$, then the  diagonal entry of $B_{G,p}$ corresponding to $v$ is $k-\ell_v$. The  row sum of $B_{G,p}[U]$ corresponding to $v$ is
$$
k-\ell_v+d_{G[U]}(v)=k-\ell_v+\Delta-1-(|W|-\ell_v)=\Delta-2.
$$
If $v\in W$ has $r_v$ nonneighbors in $U$, then the row sum of $B_{G,p}[W]$ corresponding to $v$ is
$$
k-r_v+d_{G[W]}(v)=k-r_v+\Delta-(|U|-r_v)=n-\Delta+2.
$$
Hence, $\{\{p\},U,W\}$ is an equitable partition for $B_{G,p}$, with quotient matrix $Q_{\Delta-2}$ from Claim \ref{claim12}, which leads to
$$
\rho(B_{G,p})=\lambda_{0,\Delta-2}=\lambda_0.
$$

Recall that $X=(x_v)_{v\in V(G)}$ is the Perron vector of $G$. Set
$$
\rho=\rho(G),\qquad M=\max_{v\in T}x_v,\qquad m=\min_{v\in T}x_v.
$$
Since $S_G=\{p\}$ here,  the hypotheses of Lemma \ref{lemma6} are automatic. Moreover,
$$
\Delta-\frac{2}{n}=\frac{2e(G)}{n}\le\rho\le\Delta,
$$
implies $\rho=n+O(1)$. Applying Lemma \ref{lemma6} we get
$$
M-m=O(n^{-3/2}),\qquad m,M=O(n^{-1/2}).
$$
Consider the Perron equation at $p$, we have
$$
\frac{\Delta-2}{\rho}m\le x_p=\frac{1}{\rho}\sum_{w\in U}x_w\le\frac{\Delta-2}{\rho}M.
$$
Since $(\Delta-2)/\rho=1+O(n^{-1})$, it follows that
$$
x_w=m+O(n^{-3/2})\quad\text{for all}\quad  w\in V(G).
$$
Now subtracting the Perron equations at $u$ and $v$, we have
$$
(\rho+1)(x_u-x_v)=\sum_{w\in V(G)\setminus N_G[v]}x_w-\sum_{w\in V(G)\setminus N_G[u]}x_w=O(n^{-3/2})
$$
for all $u,v\in T$.
Consequently,
$$
M-m=O(n^{-5/2}).
$$

Similarly as in the proof of Claim \ref{claim9}, we have $|\mathcal L(G,p)|\le k^2-1$.
Notice that the vector $X$ is also the Perron vector of $B_G$. Since $B_G-B_{G,p}$ is positive semidefinite, by Lemma \ref{tv} and the Rayleigh quotient we have
\begin{align*}
0\le \rho(G)+k-\lambda_0
&=\rho(B_G)-\rho(B_{G,p})\\
&\le X^T(B_G-B_{G,p})X\\
&=\sum_{(u,v)\in\mathcal L(G,p)}(x_u-x_v)^2
=O(n^{-5}).
\end{align*}
This proves the first assertion.

For $G_2$, let $Z$ be the  Perron vector of $B_{G_2,u}$. Recall that $\{\{u\}, \{v,w\}, V(G_2)\setminus\{u,v,w\}\}$ is an equitable partition. We may assume   the entries of $Z$  corresponding to vertices from $\{u\}$, $\{v,w\}$, and $V(G_2)\setminus\{u,v,w\}$ to be $z_0,z_1,z_2$, respectively. By the first row and the last row of the quotient eigenvalue equation $Q_2(z_0,z_1,z_2)^T=\lambda_0(z_0,z_1,z_2)^T$, we have
$$(\lambda_0-k)z_0=2z_1\quad\text{and}\quad
z_1-z_2=\frac{\lambda_0-n}{2}z_2.
$$
By Claim \ref{claim11}, we get
$$
z_1=\left(1-\frac1n+O(n^{-2})\right)z_2,\qquad z_0=O(n^{-1})z_2.
$$
Since $Z$ is unit, then
$$
1=z_0^2+2z_1^2+(n-3)z_2^2,
$$
which leads to $z_2^2=n^{-1}+O(n^{-2})$. Since $(\lambda_0-n)/2=-n^{-1}+O(n^{-2})$, it follows that
$$
(z_1-z_2)^2=\frac1{n^3}+O(n^{-4}).
$$
Recall that $B_{G_2}-B_{G_2,u}$ is the sum of exactly $2(k-1)$ matrices of the form $(e_a-e_b)(e_a-e_b)^T$. The Rayleigh quotient gives
\[
\rho(G_2)+k\ge Z^TB_{G_2}Z=\lambda_0+2(k-1)(z_1-z_2)^2=\lambda_0+\frac{2(k-1)}{n^3}+O(n^{-4}).
\]
Therefore, for sufficiently large $n$, $\rho(G_2)+k\ge\lambda_0+(k-1)/n^3$, which proves the second assertion.
\end{proof}

We now complete the proof of Theorem \ref{th1}.   By Lemma \ref{lemma3}, we have $d(G)\in \mathcal{D}$.
For each candidate degree sequence  $d$ in $\mathcal D$, choose a vertex $p\in S'$.
Since $s\le k-1$ and $\rho(G)>n-2k+1$, by Claim \ref{claim9} we have
\begin{equation}\label{eq29}
\rho(G)+k=\rho(B_{G,p})+O(n^{-3}).
\end{equation}

Recall the estimates in Table \ref{tab:quotient-root-estimates}, obtained from Claims \ref{claim10}
and \ref{claim11}.  By Claim \ref{claim11},  if
\[
d(G)\notin\{(\Delta^{n-1},1),(\Delta^{n-1},\Delta-1)\},
\]
then
\begin{equation}\label{eq:root-gap-one}
\rho(B_{G,p})\le\lambda_{0,1}-\frac1{n^2}+O(n^{-3});
\end{equation}
 if
\[
d(G)\notin\{(\Delta^{n-1},\Delta-1),(\Delta^{n-1},2),(\Delta^{n-1},\Delta-2)\}\cup\{(\Delta^{n-s},s,(s-1)^{s-1}):1\le s<k\},
\]
then
\begin{equation}\label{eq:root-gap-two}
\rho(B_{G,p})\le\lambda_{0,2}-\frac1{n^2}+O(n^{-3}).
\end{equation}

Suppose first that $\Delta$ is odd. By Claim \ref{claim11}, $\lambda_{0,1}=\lambda_{0,\Delta-1}$.
Notice that
\[
(n-1)\Delta+1\equiv n\pmod 2,\qquad (n-1)\Delta+(\Delta-1)=n\Delta-1\equiv n-1\pmod 2.
\]
There is a unique $\delta\in\{1,\Delta-1\}$ such that
$(n-1)\Delta+\delta$ is even. By Claim \ref{claim12}, there exists a graph $G_\delta\in \mathcal{C}(n,\Delta)$ with minimum degree $\delta$ such that $\rho(G_\delta)+k\ge\lambda_{0,\delta}=\lambda_{0,1}$. If
$d(G)\notin\{(\Delta^{n-1},1),(\Delta^{n-1},\Delta-1)\}$, then \eqref{eq29} and
\eqref{eq:root-gap-one} give $\rho(G)+k<\lambda_{0,1}\le\rho(G_\delta)+k$, which contradicts the
extremality of $G$.

Since $\Delta=n-k$ is odd, it follows that
\[
d(G)=
\begin{cases}
(\Delta^{n-1},1),&\text{if $n$ is even and $k$ is odd};\\
(\Delta^{n-1},\Delta-1),&\text{if $n$ is odd and $k$ is even}.
\end{cases}
\]

It remains to consider the case in which $\Delta$ is even, equivalently, $n$ and $k$ have the same parity.
The degree sums of $(\Delta^{n-1},1)$, $(\Delta^{n-1},\Delta-1)$, and
$(\Delta^{n-s},s,(s-1)^{s-1})$ for $1\le s<k$ are odd, since
\[
(n-1)\Delta+1\equiv(n-1)\Delta+(\Delta-1)\equiv(n-s)\Delta+s+(s-1)^2\equiv1\pmod 2.
\]
Hence none of these sequences is graphical. Therefore, \eqref{eq:root-gap-two} applies whenever
\[
d(G)\notin\{(\Delta^{n-1},2),(\Delta^{n-1},\Delta-2)\}.
\]

For any such $d(G)$, \eqref{eq29} and \eqref{eq:root-gap-two} give
$\rho(G)+k<\lambda_{0,2}$ for sufficiently large $n$. On the other hand, by Claim \ref{claim12}, there is a graph $G'\in \mathcal{C}(n,\Delta)$ such that $\rho(G')+k\ge\lambda_{0,2}$, which contradicts the extremality
of $G$. Consequently,
\[
d(G)\in\{(\Delta^{n-1},2),(\Delta^{n-1},\Delta-2)\}.
\]
By Claim \ref{claim11}, we have
$\lambda_{0,2}=\lambda_{0,\Delta-2}$.
If $d(G)=(\Delta^{n-1},\Delta-2)$, then by Claim \ref{claim13}, for sufficiently large $n$ we have
\begin{align*}
\rho(G)+k&\le\lambda_{0,\Delta-2}+O(n^{-5})=\lambda_{0,2}+O(n^{-5})\\
&<\lambda_{0,2}+\frac{k-1}{n^3}\le\rho(G')+k,
\end{align*}
which contradicts the extremality of $G$. Hence, we have
\[
d(G)=(\Delta^{n-1},2).
\]
This completes the proof.
\section*{Declaration on the use of AI}
A complete pre-revision version of this manuscript, including all mathematical results and their original proofs, was prepared by the authors without the use of any AI tools. During subsequent revisions, OpenAI's ChatGPT-5.6 Sol was used to assist with English-language editing of the whole paper, and improving the presentation of the proof of Claim 13. All AI-assisted revisions were independently examined, verified, and edited by the authors. The authors take full responsibility for the final content of the manuscript.

\enlargethispage{2\baselineskip}

\end{document}